\documentclass[12pt]{article}
\usepackage{amssymb}
\usepackage{color}
\usepackage[dvipsnames]{xcolor}
\usepackage{amsmath}
\usepackage{cite}
\usepackage{graphicx}
\usepackage{enumerate}
\newenvironment{proof}[1][Proof]{\noindent\textit{#1.} }{\hfill  \rule{0.5em}{0.5em}}

\newtheorem{theorem}{Theorem}
\newtheorem{lm}{Lemma}
\newtheorem{thm}[theorem]{Theorem}

\newtheorem{exmp}{Example}

\newtheorem{coro}{Corollary}

\begin{document}

\title{
Iteration and iterative equation on lattices
}

\author{
Chaitanya Gopalakrishna\,$^a$,~~
Weinian Zhang\,$^b$
\vspace{3mm}\\
$^a${\small Statistics and Mathematics Unit, Indian Statistical Institute,}
\\
{\small  R.V. College Post, Bangalore-560059, India}
\\
$^b${\small School of Mathematics, Sichuan University,}
\\
{\small Chengdu, Sichuan 610064, P. R. China}
    \vspace{0.2cm}\\
    {\small cberbalaje@gmail.com (CG),~~matzwn@126.com (WZ).}}

\date{}


\maketitle

\begin{abstract}
In this paper we investigate iteration of maps on lattices and the corresponding polynomial-like iterative equation.
Since a lattice need not have a metric space structure,
neither the Schauder fixed point theorem nor the Banach fixed point theorem is available.
Using Tarski's fixed point theorem,
we prove the existence of order-preserving solutions
on convex complete sublattices of Riesz spaces.
Further, in $\mathbb{R}^n$ and $\mathbb{R}$, special cases of Riesz space,
we discuss upper semi-continuous solutions and integrable solutions
respectively.
Finally, we indicate more special cases of Riesz space for discussion on the iterative equation.

\vskip 0.2cm

\vskip 0.2cm

\end{abstract}



\baselineskip 16pt
\parskip 10pt


\section{Introduction}

Iteration is an important operation and a standard element of most of the algorithms in the modern world.
Consider a self-map $f$ on a nonempty set $X$. The $k$-th order iterate $f^k$ is defined recursively by $f^0={\rm id}$,
the identity map on $X$, and $f^k=f\circ f^{k-1}$.
Those equations in which iteration of unknown functions is involved are called {\it iterative equations}.
In $\mathbb{R}^n$ a fundamental form of iterative equations is the following {\it polynomial-like iterative equation}
\begin{eqnarray} \label{1'}
\lambda_1f+\lambda_2f^2+\cdots+ \lambda_mf^m=F,
\end{eqnarray}
where $F$ is given and $f$ is unknown.
There are obtained many results on continuous solutions, differentiable solutions, convex solutions and
decreasing solutions (see e.g. \cite{Xu-Zhang,zhang1990,Bing-Weinian}) of \eqref{1'} for $n=1$,
and equivariant solutions (\cite{Zhang-Edinb}) for general $n$.
Equation (\ref{1'}) was also discussed by J. Tabor and M. Zoldak \cite{Tabor}
in the case that $X$ is a Banach space.

Lattice is one of the important objects of study in the order theory, which is significant in mathematics and computer science (\cite{Birkhoff, Garg}).
As defined in \cite{Szasz}, a relation $\preceq$ on a nonempty set $X$ is
called a {\it partial order} if it is
reflexive (i.e., $x\preceq x$ for all $x\in X$),
antisymmetric (i.e., $x=y$ whenever $x\preceq y$ and $y\preceq x$ in $X$),
and
transitive (i.e., $x\preceq z$ whenever  $x\preceq y$ and $y\preceq z$ in $X$).
$X$ endowed with a partial order $\preceq$
is called a {\it partially ordered set} (or simply a {\it poset}).
 For a subset $E$ of the poset $X$,
 $b\in X$ is called an {\it upper bound} (resp. a {\it lower bound}) of $E$
 if $x\preceq b$ (resp. $b\preceq x$) $\forall x\in E$.
 Further, $b$ is called the
 {\it least upper bound} or {\it supremum} (resp. {\it greatest lower bound} or {\it infimum}),
 denoted by $\sup_X E$ (resp. $\inf_X E$),
 if $b$ is an  upper bound (resp.  lower bound) of $E$ and
every upper bound (resp. lower bound) $z$ of $E$ satisfies
$b\preceq z$ (resp. $z\preceq b$).
A poset $X$ is called a {\it lattice} if $\sup_X\{x,y\}$, $\inf_X\{x,y\}\in X$ for every $x,y \in X$.
As defined in \cite{Luxemburg1971}, a real vector space $X$ with addition $+$ and scalar multiplication $\cdot$
is called an {\it ordered vector space}
if $X$ is a poset in a partial order $\preceq$ such that
\begin{description}
\item{\bf (i)} $x\preceq y$ implies $x+z\preceq y+z$ for all $z\in X$, and

\item{\bf (ii)} $x\succeq 0$ implies $\alpha\cdot x\succeq 0$  for all real number  $\alpha\ge 0$.
\end{description}
$X$ is called a {\it Riesz space} (or a {\it vector lattice}) if $X$
is both an ordered vector space and a lattice.
For convenience, we use $\alpha x$ to denote $\alpha\cdot x$ and $(X \preceq)$ to denote a lattice or a Riesz space $X$ in the partial order $\preceq$.
An important Riesz space is the real vector space
 $\mathcal{G}([a,b],\mathbb{R})$,
  consisting of all real-valued functions on
the compact interval $[a,b]$ in the partial order $\preceq$ defined by $f \preceq g$ if $f(x)\le g(x)$ for all $x\in [a,b]$.
Another example is
the real vector space $\mathbb{R}^n$ in the {\it Lexicographic order} (or the {\it  dictionary order}) $\preceq$ defined by $(x_1, x_2,\dots,x_n)\preceq(y_1,y_2,\ldots, y_n)$ if $x_k<y_k$ for the first $k\in\{1,2,\ldots,n\}$ such that $x_k\ne y_k$.
Special attentions have been paid to the study of Riesz spaces because of their extensive applications
to algebra \cite{Luxemburg1971,Sadovskii}, measure theory \cite{Schaefer,Fremlin}, functional analysis \cite{Kalton,Meyer},  operator theory \cite{Zaanen1997}, economics \cite{Aliprantis-Brown,Aliprantis}
and modelling of switching electronic circuits \cite{Stank,Whitesiti}.

In this paper we investigate equation (\ref{1'}) on a Riesz space $X$.
After discussing iteration of
order-preserving maps $f:X\to X$ (i.e., $f(x)\preceq f(y)$ whenever $x\preceq y$ in $X$)
in section 2,
we find in section 3 conditions under which equation (\ref{1'}) has an order-preserving solution
on convex complete sublattices of $X$.
Since a Riesz space need not even have the metric space structure,
neither the Schauder fixed point theorem
nor the Banach fixed point theorem is available.
This difficulty is overcome by using Tarski's fixed point theorem \cite{Tarski}.
Moreover,
we also give some results on uniqueness of solutions and remark for order-reversing cases.
In section 4
we additionally discuss
upper-semi-continuous (abbreviated as USC) solutions of \eqref{1'} on convex complete sublattices
in the special case that $X$ is the Euclidean space $\mathbb{R}^n$.
In section 5
we give existence and uniqueness for integrable solutions of \eqref{1'} on compact intervals
in the special case that $X$ is the real line $\mathbb{R}$.
Finally, in Section 6,
we demonstrate our results with examples and indicate more special cases of Riesz spaces
for further discussion.


\section{Iteration of order-preserving self-maps}\label{sec2}

In this section, we discuss iteration
of order-preserving self-maps on a Riesz space $X$.
Let $\preceq$ denote the partial order in $X$.
As defined in \cite{Szasz}, $X$ being a lattice is said to be {\bf (i)} {\it join-complete}
if $\sup_X E \in X$ for every nonempty subset $E$ of $X$; {\bf (ii)} {\it meet-complete} if $\inf_X E \in X$
for every nonempty subset $E$ of $X$; {\bf (iii)} {\it complete} if $X$ is both join- and meet-complete.
$X$ is said
to be {\it simply ordered} (or a {\it chain}) if at least one of the relations $x\preceq y$ and $y\preceq x$ hold whenever $x,y \in X$.
Further,
a subset $E$ of $X$ is said to be {\bf (i)} a {\it sublattice} of $X$ if $E$ itself is a lattice with respect to the order inherited from $X$ (i.e., $\sup_L\{x,y\}$, $\inf_L\{x,y\}\in E$ for every $x,y \in E$); {\bf (ii)}
 {\it convex} if $\{z\in X: x\preceq z \preceq y\}\subseteq E$ whenever $x \preceq y$ in $E$;
 {\bf (iii)} a {\it complete sublattice} of $X$ if it is a complete lattice with respect to the order inherited from $X$ (i.e., $\sup_L Y, \inf_L Y\in E$ for every subset $Y$ of $E$);
 {\bf (iv)} a {\it convex complete sublattice} of $X$ if it is a complete sublattice of $X$ and convex.
Complete lattices, being a special subclass of lattices, have been studied extensively because of its applications to various other fields of mathematics (\cite{Atsumi1966,Kertz2000,Gratzer2014,Ronse1990}).

Let $\mathcal{F}(X)$  and $\mathcal{F}_{op}(X)$ denote the poset of all self-maps  and order-preserving self-maps of $X$
 respectively in the {\it pointwise partial order} $\trianglelefteq$ defined by
 \begin{eqnarray}\label{01}
 f\trianglelefteq g\quad \text{if}\quad f(x)\preceq g(x)\ \ \forall x\in X.
 \end{eqnarray}
As in \cite{Glazowska},
for $f, g\in \mathcal{F}(X)$, we say that $f$ {\it subcommutes} with $g$
if
$$
f\circ g \trianglelefteq g\circ f.
$$
For each $f\in \mathcal{F}(X)$, let $\sup f:=\sup\{f(x): x\in X\}$ and  $\inf f:=\inf\{f(x): x\in X\}$ whenever they exist.

\begin{lm}\label{Lm1}
The following assertions are true:
\begin{description}
\item[(i)] Both $\mathcal{F}(X)$ and $\mathcal{F}_{op}(X)$ are lattices in the partial order $\trianglelefteq$.
				
\item[(ii)] If $K$ is a convex complete sublattice of $X$, then $\mathcal{F}_{op}(K)$ is a complete lattice in the partial order $\trianglelefteq$.
\end{description}
\end{lm}

\begin{proof}
The proof of {\bf (i)} is simple.
In order to prove result {\bf (ii)}, suppose that $K$ is a convex complete sublattice of $X$ and $\mathcal{E}$ be an arbitrary nonempty subset of $\mathcal{F}_{op}(K)$. Then the maps $\phi, \Phi:K\to K$ defined by $\phi(x)=\inf\{f(x): f\in \mathcal{E}\}$ and $\Phi(x)=\sup\{f(x): f\in \mathcal{E}\}$ are $\inf_{\mathcal{F}_{op}(K)}\mathcal{E}$ and $\sup_{\mathcal{F}_{op}(K)}\mathcal{E}$, respectively. Therefore $\mathcal{F}_{op}(K)$ is a complete lattice.
\end{proof}

\begin{lm}{\rm (Tarski \cite{Tarski})}\label{L0}
Let $(X, \preceq)$ be a complete lattice and $f\in \mathcal{F}_{op}(X)$.
Then the set of all fixed points of  $f$ is a non-empty complete sublattice of $X$.
Furthermore,
$f$ has the minimum fixed point $x_*$ and the maximum fixed point $x^*$ in $X$ given by
$x_*=\inf\{x\in X: f(x)\preceq x\}$ and
$x^*=\sup\{x\in X: x\preceq f(x)\}$.
\end{lm}
	
The first part of this lemma can also be found in the expository article \cite{subra2000}.
The second part, showing that $\sup\{x\in X: x\preceq f(x)\}$ and $\inf\{x\in X: f(x)\preceq x\}$ are fixed points of $f$ thereby proving the existence of a fixed point, can also be found in the book \cite{Gratzer1978}.

\begin{thm}\label{L-thm}
Let $f, g \in \mathcal{F}_{op}(X)$. The following assertions are true:
	\begin{description}
		\item[(i)] $f^k\in \mathcal{F}_{op}(X)$ for each $k\in \mathbb{N}$.
		
		\item[(ii)] If $f\trianglelefteq g$, then
		$f^k\trianglelefteq g^k$ for each $k\in \mathbb{N}$.
		
		\item[(iii)] If $f$ subcommutes with $g$, then $f$ subcommutes with $g^k$ for each $k\in \mathbb{N}$.
		
		\item[(iv)] If $f$ subcommutes with $g$ and $f(x)\preceq g(x)$, then $f^k(x)\preceq g^k(x)$ for each $k\in \mathbb{N}$.
			\end{description}
\end{thm}

\begin{proof}
	Result	{\bf (i)} is trivial.
We prove result {\bf (ii)} by induction on $k$. Clearly, it is true for $k=1$. So, let $k>1$ and suppose that  $f, g \in \mathcal{F}_{op}(X)$ satisfy $f^{k-1}\trianglelefteq g^{k-1}$. Then for each $x\in X$, we have
	\begin{eqnarray*}
		f^k(x)&=&f(f^{k-1}(x))\\
		&\preceq &f(g^{k-1}(x))~~(\text{since}~f^{k-1}(x)\preceq g^{k-1}(x)~\text{and}~f~\text{is order-preserving})\\
		&\preceq& g(g^{k-1}(x))~~(\text{since}~f\trianglelefteq g)\\
		&=&g^k(x),
	\end{eqnarray*}
	implying that $f^k\trianglelefteq g^k$.

Result {\bf (iii)} is also proved by induction on $k$. Clearly, it is true for $k=1$. So, let $k>1$ and suppose that $f, g \in \mathcal{F}_{op}(X)$ satisfy  $f\circ g^{k-1}\trianglelefteq g^{k-1}\circ f$. Then for each $x\in X$, we have
	\begin{eqnarray*}
	f\circ g^k(x)&=&f\circ g^{k-1}(g(x))\\
	&\preceq &g^{k-1}\circ f(g(x))~~(\text{since}~f\circ g^{k-1}\trianglelefteq g^{k-1}\circ f)\\
	&=&g^{k-1}(f\circ g(x))\\
	&\preceq& g^{k-1}(g\circ f(x))~~(\text{since}~f\circ g\trianglelefteq g\circ f,~\text{and}~g^{k-1}~\text{is order-preserving by {\bf (i)}})\\
	&=&g^k\circ f(x),
\end{eqnarray*}
implying that $f\circ g^k\trianglelefteq g^k\circ f$.

Again, we prove result {\bf (iv)} by induction on $k$.
Clearly, it is true for $k=1$. So, let $k>1$ and suppose that $f, g \in \mathcal{F}_{op}(X)$ satisfy  $f\circ g\trianglelefteq g\circ f$ and $f^j(x)\preceq g^j(x)$ for $1\le j \le k-1$, where $x\in X$.
	Then we have
	\begin{eqnarray*}
		f^k(x)=f(f^{k-1}(x))&\preceq &f(g^{k-1}(x))~~(\text{since}~f~\text{is order-preserving})\\
		&\preceq&g^{k-1}(f(x))~~(\text{since}~f\circ g^{k-1}\trianglelefteq g^{k-1}\circ f~\text{by {\bf (iii)}})\\
		&\preceq &g^{k-1}(g(x))~~(\text{since}~f(x)\preceq g(x),~\text{and}~g^{k-1}~\text{is order-preserving by {\bf (i)}})\\
		&=&g^k(x),
	\end{eqnarray*}
proving the result for $k$.
\end{proof}


\section{Order-preserving solutions}\label{sec22}

In this section, we give results on the existence and uniqueness of order-preserving solutions of \eqref{1'} on convex complete sublattices $K$
of a  Riesz space $X$.
Unless explaining particularly, let $(X, \preceq)$ be a Riesz space and $K$ a convex complete sublattice of $X$.

\begin{lm}\label{L1}
	Let $\lambda>0$, $\lambda_1\le 1$ and $\lambda_k\le 0$ for $2\le k\le m$ such that $\sum_{k=1}^{m}\lambda_k=\lambda$, and
	$F\in \mathcal{F}(K)$. Then a map $f$ is a solution
	of the equation
		\begin{eqnarray} \label{1}
	\lambda_1f+\lambda_2f^2+\cdots+ \lambda_mf^m=\lambda F
	\end{eqnarray}
	in $\mathcal{F}(K)$ if and only if it is a fixed point
	of the operator $T:\mathcal{F}(K)\to \mathcal{F}(K)$ given by
	\begin{eqnarray}\label{T}
	Tf=\alpha_1f+\alpha_2f^2+\cdots+ \alpha_mf^m+\alpha F
	\end{eqnarray}
	where
$\alpha=\lambda$, $\alpha_1=1-\lambda_1$ and $\alpha_k=-\lambda_k$ for $2\le k\le m$.
\end{lm}

\begin{proof}
	Let $f$ be a solution of \eqref{1} in $\mathcal{F}(K)$.
		By using the assumptions on $\lambda$ and $\lambda_k$s, we see that
	\begin{eqnarray}\label{alphak ineuality}
	\alpha>0,\quad \alpha_k\ge 0~~ \text{for}~~1\le k\le m, \quad\text{and}\quad  \sum_{k=1}^{m}\alpha_k+ \alpha=1.
	\end{eqnarray}
	Let $\gamma:=\min K$ and $\Gamma:=\max K$, both of which exist since $(K, \preceq)$ is a complete lattice.
	Then for each $x\in K$, since $F(x), f^k(x) \in K$, we have $\gamma \preceq F(x)\preceq \Gamma$
	and $\gamma \preceq f^k(x)\le \Gamma$
	for $1\le k \le m$. This implies by \eqref{alphak ineuality} that
	\begin{eqnarray}\label{02}
	\gamma =1\cdot \gamma=\sum_{k=1}^{m}\alpha_k\gamma+\alpha \gamma\preceq \sum_{k=1}^{m}\alpha_k f^k(x)+\alpha F(x)\preceq  \sum_{k=1}^{m}\alpha_k\Gamma+\alpha \Gamma=1\cdot \Gamma=\Gamma
	\end{eqnarray}
	for each $x\in K$. i.e.,
	$\gamma \preceq Tf(x) \preceq \Gamma$, proving that $Tf(x)\in K$ for each $x\in K$, since $\gamma,\Gamma \in K$ and $K$ is convex. Therefore $Tf$ is a self-map of $K$. Hence $T$ is self-map of $\mathcal{F}(K)$.
	 Further, for each $x\in K$, we have
	\begin{eqnarray*}
		Tf(x)&=&\sum_{k=1}^{m}\alpha_kf^k(x)+\alpha F(x)\\
		&=&(1-\lambda_1)f(x)+\sum_{k=2}^{m}(-\lambda_k)f^k(x)+\lambda F(x)\\
		&=&f(x)-\sum_{k=1}^{m}\lambda_kf^k(x)+\lambda F(x)\\
		&=&f(x)-\lambda F(x)+\lambda F(x)=f(x),
	\end{eqnarray*}
	implying that $f$ is a fixed point of $T$. This proves the ``only if'' part. The proof of ``if '' part is similar.
\end{proof}

Having the above lemma, we are ready to give the following.

\begin{theorem}\label{Thm1}
Let $\lambda>0$, $\lambda_1\le 1$, and $\lambda_k\le 0$ for $2\le k\le m$ such that $\sum_{k=1}^{m}\lambda_k=\lambda$.
If  $F\in \mathcal{F}_{op}(K)$ satisfies $\frac{1}{\lambda} \sup F\in K$,
then
the set $\mathcal{S}_{op}(K)$ of all solutions of equation \eqref{1'} in $\mathcal{F}_{op}(K)$
is a non-empty complete sublattice of $\mathcal{F}_{op}(K)$.
Furthermore,
equation \eqref{1'} has the minimum solution $f_*$ and the maximum solution $f^*$ in $\mathcal{F}_{op}(K)$  given by
\begin{eqnarray*}
&&f_*=\inf\{f\in \mathcal{F}_{op}(K):  \lambda F\trianglelefteq  \lambda_1f+\lambda_2f^2+\cdots+\lambda_mf^m\},
\\
&&f^*=\sup\{f\in \mathcal{F}_{op}(K): \lambda_1f+\lambda_2f^2+\cdots+\lambda_mf^m \trianglelefteq  \lambda F\}.
\end{eqnarray*}
\end{theorem}

\begin{proof}
For $F\in \mathcal{F}_{op}(K)$, we first prove in the following two steps that
the set of all solutions of \eqref{1}  in $\mathcal{F}_{op}(K)$ is a non-empty complete sublattice of $\mathcal{F}_{op}(K)$.

\noindent
{\it Step 1.} Construct an order-preserving map $T:\mathcal{F}_{op}(K) \to \mathcal{F}_{op}(K)$.
	
	Given real numbers $\lambda$ and $\lambda_k$s as above, define a map $T$ on  $\mathcal{F}_{op}(K)$ as in \eqref{T},
	where $\alpha$ and $\alpha_k$s are chosen as in Lemma \ref{L1}.
	Then, by using the assumptions on $\lambda$ and $\lambda_k$s, we see that  $\alpha$ and $\alpha_k$s satisfy \eqref{alphak ineuality}.
Further, by a similar argument as in Lemma \ref{L1}, it follows that $Tf$ is a self-map of $K$.
	
	
		Next, to prove that $Tf$ is order-preserving, consider any $x, y\in K$ such that $x\preceq y$.  Since $F, f$ are order-preserving on $K$, we have $F(x) \preceq F(y)$ and $f^k(x)\preceq f^k(y)$ for $1\le k \le m$, implying by \eqref{alphak ineuality} that
	\begin{eqnarray}\label{Tf op}
	Tf(x)=\sum_{k=1}^{m}\alpha_kf^k(x)+\alpha F(x)\preceq \sum_{k=1}^{m} \alpha_kf^k(y)+\alpha F(y)=Tf(y).
	\end{eqnarray}	
	\noindent Therefore $Tf\in \mathcal{F}_{op}(K)$, and thus $T$ is a self-map of $\mathcal{F}_{op}(K)$.
	
	Finally, to prove that $T$ is order-preserving,
consider any $f, g \in \mathcal{F}_{op}(K)$ such that $f\trianglelefteq g$. Then by result {\bf (ii)}
 of Theorem \ref{L-thm},  we have 	$f^k\trianglelefteq g^k$ for $1\le k \le m$, implying that
	\begin{eqnarray}\label{T op}
	Tf(x)=\sum_{k=1}^{m}\alpha_kf^k(x)+\alpha F(x)\preceq \sum_{k=1}^{m} \alpha_kg^k(x)+\alpha F(x)=Tg(x)
	\end{eqnarray}
	for each $x\in K$,  i.e., $Tf\trianglelefteq Tg$. Hence $T$ is order-preserving.


\noindent
{\it Step 2.} Prove that the set of all solutions of \eqref{1}
in $\mathcal{F}_{op}(K)$ is a non-empty complete sublattice of $\mathcal{F}_{op}(K)$.
	
From Step 1
we see that $T$ is an order-preserving self-map of the lattice  $\mathcal{F}_{op}(K)$, which is complete by result {\bf (ii)} of Lemma \ref{Lm1}.
	Therefore by Lemma \ref{L0},
	the set of all fixed points of $T$ in $\mathcal{F}_{op}(K)$, and hence by Lemma \ref{L1}, the set of all solutions of \eqref{1}  in $\mathcal{F}_{op}(K)$ is a non-empty complete sublattice of $\mathcal{F}_{op}(K)$.

Now, in order to prove our result,	given $F$ as above, let $G:=\frac{1}{\lambda}F$. Then, since $\lambda>0$, clearly $G$ is order-preserving on $K$. Also, since $F$ is a self-map of the complete lattice $K$, we have $\sup F, \inf F \in K$, and therefore
	\begin{eqnarray}\label{G}
	\inf F\preceq \frac{1}{\lambda}\inf F\preceq \frac{1}{\lambda}F(x)\preceq \frac{1}{\lambda}\sup F,\quad \forall x\in K.
	\end{eqnarray}
	This implies that $G(x) \in K$ for all $x\in K$, since by assumption $\frac{1}{\lambda} \sup F\in K$ and $K$ is convex.
	 Therefore $G\in \mathcal{F}_{op}(K)$.  This implies by the above part that the
	 set of all solutions of the equation
	\begin{eqnarray} \label{2}
	\lambda_1f+\lambda_2f^2+\cdots+ \lambda_mf^m=\lambda G
	\end{eqnarray}
	 in $\mathcal{F}_{op}(K)$ is a non-empty complete sublattice of $\mathcal{F}_{op}(K)$.
	
	 In particular, \eqref{1'} has the minimum solution $f_*$ and the maximum solution $f^*$ in $\mathcal{F}_{op}(K)$, which are in fact $\min \mathcal{S}_{op}(K)$ and $\max \mathcal{S}_{op}(K)$, respectively.
	 Further, by Lemma \ref{L0}, we have $f_*=\inf\{f\in \mathcal{F}_{op}(K):  Tf \trianglelefteq  f\}$ and $f^*=\sup\{f\in \mathcal{F}_{op}(K): f\trianglelefteq  Tf\}$.
This completes the proof.
\end{proof}

The following theorem is devoted to uniqueness of solutions.

\begin{thm}\label{Thm3}
		Let $\lambda_1>0$, $\lambda_k\ge 0$ for $2\le k\le m$, and $F\in \mathcal{F}_{op}(K)$.
	\begin{description}
		\item[(i)] If $f, g \in \mathcal{F}_{op}(K)$ are solutions of \eqref{1'} on $K$ such that 	$f\trianglelefteq g$, then $f=g$ on $K$.
		
		\item[(ii)]  If $f, g \in \mathcal{F}_{op}(X)$ are solutions of \eqref{1'} on $K$ such that $f$ subcommutes with $g$ and $f(x)\prec g(x)$ for some $x\in K$, then $f=g$ on $K$.
		
		\item[(iii)] If $K$ is a chain in the order $\preceq$, and $f, g \in \mathcal{F}_{op}(K)$ are solutions of \eqref{1'} on $K$ such that $f\circ g=g\circ f$, then $f=g$ on $K$.
	\end{description}

\end{thm}

\begin{proof}
	Let  $f, g \in \mathcal{F}_{op}(K)$ be solutions of \eqref{1'} on $K$ such that $f\trianglelefteq g$, and  suppose that $f\ne g$ on $K$. Then there exists $x\in K$ such that $f(x)\prec g(x)$, and by result {\bf (ii)} of Theorem \ref{L-thm}, we have $f^k\trianglelefteq g^k$, implying that $f^k(x)\preceq g^k(x)$ for $2\le k\le m$. Therefore, by \eqref{1'}, we have
	\begin{eqnarray}\label{contr}
	F(x)=\sum_{k=1}^{m}\lambda_kf^k(x)\prec\sum_{k=1}^{m}\lambda_kg^k(x)=F(x),
	\end{eqnarray}
since $\lambda_1>0$ and $\lambda_k\ge 0$ for $2\le k\le m$. This is a contradiction. Hence $f=g$ on $K$, proving result {\bf (i)}.


In order to prove result {\bf (ii)}, consider any solutions  $f, g \in \mathcal{F}_{op}(K)$ of \eqref{1'} such that $f\circ g\trianglelefteq g\circ f$ and $f(x)\prec g(x)$ for some $x\in K$. Then by result {\bf (iv)} of Theorem \ref{L-thm} we have $f^k(x)\preceq g^k(x)$ for $2\le k \le m$, and therefore  we arrive at \eqref{contr},
since $\lambda_1>0$ and $\lambda_k\ge 0$ for $2\le k\le m$. This is a contradiction. Hence $f=g$ on $K$.

Finally, to prove result {\bf (iii)},
consider any solutions  $f, g \in \mathcal{F}_{op}(K)$ of \eqref{1'} such that $f\circ g=g\circ f$, and
	suppose that $f\ne g$ on $K$. Then there exists $x\in K$ such that $f(x)\ne g(x)$, implying that either $f(x)\prec g(x)$ or $g(x)\prec f(x)$, because $K$ is a chain. In any case, by result {\bf (ii)} we have $f=g$ on $K$.
\end{proof}

Note that the condition $f\trianglelefteq g$ and the condition $f\circ g=g\circ f$,
assumed in results {\bf (i)} and {\bf (iii)} of Theorem \ref{Thm3} respectively, are independent.
If fact,
maps $f, g:[0, 1]\to [0,1]$ defined by $f(x)=\frac{x^3}{2}$ and $g(x)=x^2$
satisfy that $f, g \in \mathcal{F}_{op}([0,1])$ and $f\trianglelefteq g$.
However, $(f\circ g)(\frac{1}{2})=\frac{1}{128}\ne \frac{1}{256}=(g\circ f)(\frac{1}{2})$,
implying that $f\circ g \ne g\circ f$.
Another example is the pair of maps $f, g:[0, 1]\to [0,1]$ defined by $f(x)={\rm id}$ and
 \begin{eqnarray*}
 	g(x)=\left\{\begin{array}{cll}
 		0&\text{if}&0\le x\le \frac{1}{4},\\
 		2x-\frac{1}{2}&\text{if} &\frac{1}{4}\le x \le \frac{3}{4},\\
 		1&\text{if}&\frac{3}{4}\le x\le 1.
 	\end{array}\right.
 \end{eqnarray*}
Clearly, $f, g \in \mathcal{F}_{op}([0,1])$ and $f\circ g = g\circ f$
but $f\ntrianglelefteq g$ because $f(\frac{3}{4})<g(\frac{3}{4})$ and $f(\frac{1}{4})>g(\frac{1}{4})$.
This independence shows
that neither {\bf (i)} implies {\bf (iii)} nor {\bf (iii)} implies {\bf (i)} in Theorem \ref{Thm3} even if $K$ is a chain.
We also note that
the proof of result {\bf (iii)} in Theorem \ref{Thm3} is not valid when $K$ is not a chain in $\mathbb{R}^n$,
i.e.,
the assumption $K$ is a chain made in result {\bf (iii)} of Theorem \ref{Thm3} cannot be dropped.
In fact,
in that case, the inequality $x\ne y$ need not imply that either $x\prec y$ or $y\prec x$.
For example, in the lattice $[0,1]^2$, which will be considered in Example \ref{E1},
we have $(1,0)\ne (0,1)$, but neither $(1,0)\prec (0,1)$ nor $(0,1)\prec (1,0)$ hold.


Remark that the current approach with the map $T$ defined in \eqref{T}, employed in Theorem \ref{Thm1},
cannot be used to solve \eqref{1'} if
$F\in \mathcal{F}_{or}(K)$, the complete lattice of all {\it order-reversing} self-maps of $K$ in the partial order $\trianglelefteq$.
In fact,
in the case that
$F\in \mathcal{F}_{or}(K)$,
assuming that $\lambda:=\sum_{k=1}^{m}\lambda_k<0$, $\lambda_1\le 1$ and $\lambda_k\le 0$ for all
$2\le k\le m$,
we get that
the map $T$ is order-preserving on $\mathcal{F}_{op}(K)$,
but
$Tf$ is not necessarily a self-map of $K$ for $f\in \mathcal{F}_{op}(K)$
because an inequality of the form \eqref{02} is not satisfied
for the reason that the inequality
$\gamma\preceq F(x)\preceq\Gamma$ implies that $\alpha\Gamma\preceq \alpha F(x)\preceq\alpha\gamma$.
In the other case, i.e.,
$F\in \mathcal{F}_{or}(K)$
assuming that $\lambda>0$,
$\lambda_1>1$ or $\lambda_k>0$ for some
$2\le k\le m$,
we see that $Tf$ is not necessarily order-preserving on $K$ for $f\in \mathcal{F}_{op}(K)$
because
the corresponding function $\alpha F$, $\alpha_1 f$ or $\alpha_k f^k$ in the sum
$Tf=\alpha_1f+\alpha_2f^2+\cdots\alpha_mf^m+\alpha F$
is not order-preserving.
We do not consider the case that $\lambda=0$, where $F$ is not involved in $T$.

Besides, the above remarked approach cannot be used to seek a solution $f$ of \eqref{1'} in $\mathcal{F}_{or}(K)$
no matter whether $F$ is considered in $\mathcal{F}_{op}(K)$ or $\mathcal{F}_{or}(K)$.
In fact, Lemma \ref{L0} is not true if `$\mathcal{F}_{op}(X)$' is replaced with `$\mathcal{F}_{or}(X)$', as seen from the following example:
Let $X$ be the complete lattice $\{a,b,c,d\}$ in the partial order $\preceq$ such that
$a\preceq b\preceq d$ and $a\preceq c\preceq d$,
and $f:X\to X$ be the order-reversing map such that
$f(a)=d$, $f(b)=c$, $f(c)=b$ and $f(d)=a$. Then $f$ has no fixed points in $X$.

For the same reason, the approaches of Theorems \ref{Thm2} and \ref{Thm6} cannot be employed for other types of monotonicity.


\section{USC solutions in $\mathbb{R}^n$
 }
\label{sec4}

The above results are obtained for general convex complete sublattices $K$ of a Riesz space $X$
without any property of continuity or integrability.
Actually, we cannot define `continuity' or `semi-continuity' on a general Riesz space $X$, which
does not have a topology structure but only compatible algebraic and order structures.

In this section, we additionally consider upper semi-continuity
in the real vector space $\mathbb{R}^n$ ($n\ge 1$) of all real $n$-tuples
with the coordinate-wise addition
and the real scalar product,
which is also a Riesz space
in the order $\preceq$ that
$$
(x_1, x_2, \ldots, x_n)\preceq (y_1, y_2, \ldots, y_n)~~ \mbox{ if }
x_k\le y_k \mbox{ for } 1\le k \le n,
$$
where $\le$ denotes the usual order on $\mathbb{R}$.
%
As defined in \cite{Engelking1989}, a map $f:\mathbb{R}^n \to \mathbb{R}$ is said to be {\it USC} (abbreviation of {\it upper semi-continuous})
at $x_0 \in \mathbb{R}^n$ if
for every $\rho \in \mathbb{R}$ satisfying $f(x_0)<\rho$ there exists a neighbourhood  $U$ of $x_0$ in $\mathbb{R}^m$ such that $f(y)<\rho$ for all $y\in U$. Equivalently, $f$ is USC at $x_0$ if
$\limsup_{x \to x_0} f(x) \le f(x_0)$.  $f$ is said to be USC on $\mathbb{R}^n$ if $f$ is USC at  each point of $\mathbb{R}^n$.
A map $f=(f_1, f_2, \ldots, f_n):\mathbb{R}^n \to \mathbb{R}^n$, where $f_j:\mathbb{R}^n \to \mathbb{R}$ for $1\le j \le n$,
is said to be USC if $f_j$ is USC
for all $1\le j\le n$.
For each convex complete sublattice  $K$  of $\mathbb{R}^n$,
let
$$\mathcal{F}^{usc}_{op}(K):=\{f\in \mathcal{F}_{op}(K): f~\text{is USC on}~K\}.$$
By Theorem \ref{Thm1}, equation \eqref{1'} has a solution $f$ in $\mathcal{F}_{op}(K)$ for each $F\in  \mathcal{F}^{usc}_{op}(K)$;
however we cannot conclude that $f$ is USC because $\mathcal{F}^{usc}_{op}(K)\subsetneq \mathcal{F}_{op}(K)$. For USC solutions
we have the following.

\begin{thm}\label{Thm2}
Let $\lambda>0$, $\lambda_1\le 1$, and $ \lambda_k\le 0$ for $2\le k\le m$ such that $\sum_{k=1}^{m}\lambda_k=\lambda$.
If $F\in \mathcal{F}^{usc}_{op}(K)$ satisfies
$\frac{1}{\lambda} \sup F\in K$,
then
the set $\mathcal{S}^{usc}_{op}(K)$ of all solutions of equation \eqref{1'} in $\mathcal{F}^{usc}_{op}(K)$
is a non-empty complete sublattice of $\mathcal{F}^{usc}_{op}(K)$.
Furthermore,
equation \eqref{1'} has the minimum solution $f_*$ and the maximum solution $f^*$ in $\mathcal{F}^{usc}_{op}(K)$ given by
\begin{eqnarray*}
&&f_*=\inf\{f\in \mathcal{F}^{usc}_{op}(K):  \lambda F\trianglelefteq  \lambda_1f+\lambda_2f^2+\cdots+\lambda_mf^m\},
\\
&&f^*=\sup\{f\in \mathcal{F}^{usc}_{op}(K): \lambda_1f+\lambda_2f^2+\cdots+\lambda_mf^m\trianglelefteq  \lambda F\}.
\end{eqnarray*}
\end{thm}

\begin{proof}
		Let $F\in \mathcal{F}^{usc}_{op}(K)$ be arbitrary.
		
			\noindent {Step 1.} Construct an order-preserving map $T:\mathcal{F}^{use}_{op}(K) \to \mathcal{F}^{usc}_{op}(K)$.

		Given real numbers $\lambda$ and $\lambda_k$s as above, let $G:=\frac{1}{\lambda}F$ and define a map $T$ on  $\mathcal{F}^{usc}_{op}(K)$ by
		\begin{eqnarray}\label{T1}
		Tf=\alpha_1f+\alpha_2f^2+\cdots+ \alpha_mf^m+\alpha G,
		\end{eqnarray}
		where $\alpha$ and $\alpha_k$s are chosen as in Lemma \ref{L1}. 	Then, by using the assumptions on $\lambda$ and $\lambda_k$s, we see that  $\alpha$ and $\alpha_k$s satisfy \eqref{alphak ineuality}. Also, since $\lambda>0$, clearly $G$ is an order-preserving USC map on $K$. Further, since $F$ is a self-map of the complete lattice $K$, we have $\inf F \in K$, and therefore \eqref{G} is satisfied. 	This implies that $G(x) \in K$  for all $x\in K$, since by assumption $\frac{1}{\lambda} \sup F\in K$ and $K$ is convex.
		Therefore $G\in \mathcal{F}^{usc}_{op}(K)$.
		
			Consider an arbitrary $f\in \mathcal{F}^{usc}_{op}(K)$. 	Let $f^k:=(f_{k1}, f_{k2}, \ldots,f_{kn})$, where each  $f_{kj}:K\to \mathbb{R}$ is a coordinate map of $f^k$ for $1\le j\le n$ and $1\le k \le m$.
			Similarly, let $G:=(G_1, G_2, \ldots, G_n)$ and $Tf:=(H_1,H_2,\ldots,H_n)$. Then $H_j=\sum_{k=1}^{m}\alpha_k f_{kj}+\alpha G_j$ is USC  on $K$ for $1\le j \le n$, being the nonnegative linear combination of USC maps $f_{kj}$ and $G_j$ for $1\le j\le n$ and $1\le k \le m$. Therefore $Tf$ is USC on $K$.
			Also, since $\mathcal{F}^{usc}_{op}(K)\subseteq \mathcal{F}_{op}(K)$, by using Step 1  of the proof of Theorem \ref{Thm1}, we have $Tf \in \mathcal{F}_{op}(K)$ and $T$ is order-preserving.
		 Therefore $T$ is an order-preserving self-map on $\mathcal{F}^{usc}_{op}(K)$.

				
\noindent
{Step 2.} Prove that  $\mathcal{F}^{use}_{op}(K)$ is a complete lattice in the partial order $\trianglelefteq$.
		
		Consider an arbitrary subset $\mathcal{E}$ of $\mathcal{F}^{use}_{op}(K)$. If $\mathcal{E}=\emptyset$, then the constant map $\Psi:K\to K$ defined by $\Psi(x)=\max K$
is the infimum of $\mathcal{E}$ in $\mathcal{F}^{use}_{op}(K)$. If $\mathcal{E}\ne \emptyset$, then the map $\phi:K\to K$ defined by
$\phi(x)=\inf\{f(x): f \in \mathcal{E}\}$ is the infimum
of $\mathcal{E}$ in $\mathcal{F}^{use}_{op}(K)$.  Thus every subset of $\mathcal{F}^{use}_{op}(K)$ has the infimum in $\mathcal{F}^{use}_{op}(K)$. Therefore by Lemma $14$ of \cite{Gratzer1978},
which says that if every subset of a poset $P$ has the infimum in $P$ then $P$ is complete,
we know that $\mathcal{F}^{use}_{op}(K)$ is a complete lattice.


\noindent
{\it Step 3.} Prove that	
$\mathcal{S}^{usc}_{op}(K)$ is a non-empty complete sublattice of $\mathcal{F}^{usc}_{op}(K)$.

From Steps 1 and 2, we see that $T$ is an order-preserving self-map of the complete lattice  $\mathcal{F}^{usc}_{op}(K)$. Hence, by Lemma \ref{L0}, the set of all fixed points of $T$ in $\mathcal{F}^{usc}_{op}(K)$ is a non-empty complete sublattice of $\mathcal{F}^{usc}_{op}(K)$. This implies by Lemma \ref{L1} that the set of all solutions of  \eqref{2}, and hence that of \eqref{1'} in $\mathcal{F}^{usc}_{op}(K)$ is a non-empty complete sublattice of $\mathcal{F}^{usc}_{op}(K)$,  because $G=\frac{1}{\lambda}F$.
That is, $\mathcal{S}^{usc}_{op}(K)$
is a non-empty complete sublattice of $\mathcal{F}^{usc}_{op}(K)$.

In particular, \eqref{1'} has the minimum solution $f_*$ and the maximum solution $f^*$ in $\mathcal{F}^{usc}_{op}(K)$, which are in fact $\min \mathcal{S}^{usc}_{op}(K)$ and $\max \mathcal{S}^{usc}_{op}(K)$, respectively.
		 Further, by Lemma \ref{L0}, we have $f_*=\inf\{f\in \mathcal{F}_{op}^{usc}(K):  Tf \trianglelefteq  f\}$ and $f^*=\sup\{f\in \mathcal{F}_{op}^{usc}(K): f\trianglelefteq  Tf\}$.
		This completes the proof.
\end{proof}

We have the following results on uniqueness of solutions.

\begin{coro}\label{Thm5}
Let $\lambda_1>0$, $\lambda_k\ge 0$ for $2\le k\le m$, and $F\in \mathcal{F}^{usc}_{op}(K)$.
	\begin{description}
		\item[(i)] If $f, g \in \mathcal{F}^{usc}_{op}(K)$ satisfy \eqref{1'} on $K$ such that 	$f\trianglelefteq g$, then $f=g$ on $K$.
		
			\item[(ii)]  If $f, g \in \mathcal{F}^{usc}_{op}(X)$ are solutions of \eqref{1'} on $K$ such that $f$ subcommutes with $g$ and $f(x)\prec g(x)$ for some $x\in K$, then $f=g$ on $K$.
		
	\item[(iii)] 	If $f, g \in \mathcal{F}^{usc}_{op}(K)$ satisfy \eqref{1'} on $K$ such that $f\circ g=g\circ f$, then $f=g$ on $K$.
	\end{description}
\end{coro}

\begin{proof}
	Follows from Theorem \ref{Thm3}, since $\mathcal{F}_{op}^{usc}(K) \subseteq \mathcal{F}_{op}(K)$.
\end{proof}

\section{Integrable solutions in $\mathbb{R}$}\label{S5}

In the section we consider integrability of solutions on the real line $\mathbb{R}$,
which is an ordered Riesz space in the usual addition, multiplication and order $\le$.
We give results on the existence and uniqueness of integrable solutions of \eqref{1'} on
the compact interval $[a,b]$,
which is precisely a convex compete sublattice of $\mathbb{R}$ in the usual order $\le$.
As defined in \cite{royden1988}, a map $f:[a,b]\to \mathbb{R}$ is called {\it Lebesgue measurable} (or simply {\it measurable})
if $\{x\in [a,b]: f(x)>\rho\}$ is Lebesgue measurable for each $\rho\in \mathbb{R}$.
A measurable function $f:[a,b]\to \mathbb{R}$ is said to be {\it $p$-integrable},
 where $1\le p<\infty$,
if $|f|^p$ is Lebesgue integrable, i.e.,  $\int_{a}^{b}|f|^p d\mu<\infty$.
Let
$$
L_p([a,b]):=\{f\in \mathcal{F}([a,b]): f~\text{is  measurable and}~p\text{-integrable on}~[a,b]\}
$$
and $\mathcal{F}_{op}^{p}([a,b]):=\mathcal{F}_{op}([a,b])\cap L_p([a,b])$.

\begin{thm}\label{Thm6}
	Let $\lambda>0$, $\lambda_1\le 1$, and $ \lambda_k\le 0$ for $2\le k\le m$ such that $\sum_{k=1}^{m}\lambda_k=\lambda$.
	If $F\in \mathcal{F}_{op}^{p}([a,b])$ satisfies
	$\frac{1}{\lambda} \sup F\in [a,b]$,
	then
	the set $\mathcal{S}^{p}_{op}([a,b])$ of all solutions of equation \eqref{1'} in $\mathcal{F}_{op}^{p}([a,b])$
	is a non-empty complete sublattice of $\mathcal{F}_{op}^{p}([a,b])$.
	Furthermore,
	equation \eqref{1'} has the minimum solution $f_*$ and the maximum solution $f^*$ in $\mathcal{F}_{op}^{p}([a,b])$ given by
	\begin{eqnarray*}
		&&f_*=\inf\{f\in \mathcal{F}_{op}^{p}([a,b]):  \lambda F\trianglelefteq  \lambda_1f+\lambda_2f^2+\cdots+\lambda_mf^m\},
		\\
		&&f^*=\sup\{f\in \mathcal{F}_{op}^{p}([a,b]): \lambda_1f+\lambda_2f^2+\cdots+\lambda_mf^m\trianglelefteq  \lambda F\}.
	\end{eqnarray*}
\end{thm}

\begin{proof}
	Let $F\in \mathcal{F}_{op}^{p}([a,b])$ be arbitrary.
	
	\noindent {Step 1.} Construct an order-preserving map $T:\mathcal{F}_{op}^{p}([a,b]) \to \mathcal{F}_{op}^{p}([a,b])$.

	Given real numbers $\lambda$ and $\lambda_k$s as above, let $G:=\frac{1}{\lambda}F$ and define a map $T$ on  $\mathcal{F}_{op}^{p}([a,b])$ as in \eqref{T1},
	where $\alpha$ and $\alpha_k$s are chosen as in Lemma \ref{L1}. 	Then, by using the assumptions on $\lambda$ and $\lambda_k$s, we see that  $\alpha$ and $\alpha_k$s satisfy \eqref{alphak ineuality}. Also, since $\lambda>0$, clearly $G$ is an order-preserving  measurable  $p$-integrable map on $[a,b]$. Further, since $F$ is a self-map of the complete lattice $[a,b]$, we have $\inf F \in [a,b]$, and therefore \eqref{G} is satisfied for $\le$. 	This implies that $G(x) \in [a,b]$  for all $x\in [a,b]$, since by assumption $\frac{1}{\lambda} \sup F\in [a,b]$ and $[a,b]$ is convex.
	Therefore $G\in \mathcal{F}_{op}^{p}([a,b])$.
	
	Consider an arbitrary $f\in \mathcal{F}_{op}^{p}([a,b])$. 	By assumption, $f$ is  measurable. Also, since $f^k$ is order-preserving, it is  measurable on $[a,b]$ for $2\le k\le m$. Further, since $f$ is $p$-integrable on $[a,b]$, so is $f^k$ for $2\le k\le m$. Therefore $Tf$ is a  measurable $p$-integrable map on $[a,b]$, being the nonnegative linear combination of  measurable $p$-integrable maps $f^k$ and $G$ for $1\le k\le m$. 	Also, since $\mathcal{F}^{p}_{op}([a,b])\subseteq \mathcal{F}_{op}([a,b])$, by using Step 1  of the proof of Theorem \ref{Thm1}, we have $Tf \in \mathcal{F}_{op}([a,b])$ and $T$ is order-preserving.
	Therefore $T$ is an order-preserving self-map on $\mathcal{F}^{p}_{op}([a,b])$.
	
	\noindent
	{Step 2.} Prove that  $\mathcal{F}^{p}_{op}([a,b])$ is a complete lattice in the partial order $\trianglelefteq$.
	
	Consider an arbitrary subset $\mathcal{E}$ of $\mathcal{F}^{p}_{op}([a,b])$. We discuss in the two cases.
	
	\noindent {\bf Case (i):} If $\mathcal{E}=\emptyset$, then the constant map $\Psi:[a,b]\to [a,b]$ defined by $\Psi(x)=b$
	is the infimum of $\mathcal{E}$ in $\mathcal{F}^{p}_{op}([a,b])$.
	
	\noindent {\bf Case (ii):} If $\mathcal{E}\ne \emptyset$, then we assert that the map $\phi:[a,b]\to [a,b]$ defined by
	$\phi(x)=\inf\{f(x): f \in \mathcal{E}\}$ is the infimum
	of $\mathcal{E}$ in $\mathcal{F}^{p}_{op}([a,b])$.  Clearly, $\phi$ is the infimum of $\mathcal{E}$ in $\mathcal{F}_{op}([a,b])$.
	Also, for each $\rho \in \mathbb{R}$, we have
	\begin{eqnarray*}
	\{x\in [a,b]: \phi(x)>\rho\}=\bigcup_{f\in \mathcal{E}}\{x\in [a,b]:f(x)>\rho\}.
	\end{eqnarray*}
Therefore the measurability of maps $f\in \mathcal{E}$ implies that of $\phi$. Further,
being a bounded measurable map on the measurable set $[a,b]$ of finite measure,
$|f|^p$ is integrable on $[a,b]$
 by Proposition 3 of \cite[p.79]{royden1988}.
%
Therefore $f$ is $p$-integrable on $[a,b]$. Hence $\phi \in \mathcal{F}^{p}_{op}([a,b])$.

Thus every subset of $\mathcal{F}^{p}_{op}([a,b])$ has the infimum in $\mathcal{F}^{p}_{op}([a,b])$.
Therefore, by Lemma $14$ of \cite{Gratzer1978},
which says that if every subset of a poset $P$ has the infimum in $P$ then $P$ is complete,
	we know that $\mathcal{F}^{p}_{op}([a,b])$ is a complete lattice.

	
	\noindent
	{\it Step 3.} Prove that	
	$\mathcal{S}^{p}_{op}([a,b])$ is a non-empty complete sublattice of $\mathcal{F}^{p}_{op}([a,b])$.

	From Steps 1 and 2, we see that $T$ is an order-preserving self-map of the complete lattice  $\mathcal{F}^{p}_{op}([a,b])$. Hence, by Lemma \ref{L0}, the set of all fixed points of $T$ in $\mathcal{F}^{p}_{op}([a,b])$ is a non-empty complete sublattice of $\mathcal{F}^{p}_{op}([a,b])$. This implies by Lemma \ref{L1} that the set of all solutions of  \eqref{2}, and hence that of \eqref{1'} in $\mathcal{F}^{p}_{op}([a,b])$ is a non-empty complete sublattice of $\mathcal{F}^{p}_{op}([a,b])$,  because $G=\frac{1}{\lambda}F$.
	That is, $\mathcal{S}^{p}_{op}([a,b])$
	is a non-empty complete sublattice of $\mathcal{F}^{p}_{op}([a,b])$.
	
	In particular, \eqref{1'} has the minimum solution $f_*$ and the maximum solution $f^*$ in $\mathcal{F}^{p}_{op}([a,b])$, which are in fact $\min \mathcal{S}^{p}_{op}([a,b])$ and $\max \mathcal{S}^{p}_{op}([a,b])$, respectively.
	Further, by Lemma \ref{L0}, we have $f_*=\inf\{f\in \mathcal{F}_{op}^{p}([a,b]):  Tf \trianglelefteq  f\}$ and $f^*=\sup\{f\in \mathcal{F}_{op}^{p}([a,b]): f\trianglelefteq  Tf\}$.
	This completes the proof.
\end{proof}

We have the following results on uniqueness of solutions.

\begin{coro}\label{Thm7}
	Let $\lambda_1>0$, $\lambda_k\ge 0$ for $2\le k\le m$, and $F\in \mathcal{F}_{op}^{p}([a,b])$.
	\begin{description}
		\item[(i)] If $f, g \in \mathcal{F}_{op}^{p}([a,b])$ satisfy \eqref{1'} on $[a,b]$ such that 	$f\trianglelefteq g$, then $f=g$ on $[a,b]$.
		
		\item[(ii)]  If $f, g \in \mathcal{F}_{op}^{p}([a,b])$ are solutions of \eqref{1'} on $[a,b]$ such that $f$ subcommutes with $g$ and $f(x)\prec g(x)$ for some $x\in [a,b]$, then $f=g$ on $[a,b]$.
		
		\item[(iii)] 	If $f, g \in \mathcal{F}_{op}^{p}([a,b])$ satisfy \eqref{1'} on $[a,b]$ such that $f\circ g=g\circ f$, then $f=g$ on $[a,b]$.
	\end{description}
\end{coro}
\begin{proof}
Follows from Theorem \ref{Thm3}, since $\mathcal{F}_{op}^{p}([a,b]) \subseteq \mathcal{F}_{op}([a,b])$.
\end{proof}


\section{Examples and remarks}\label{sec6}

The following examples demonstrate our main theorems.


\begin{exmp}\label{E1}
{\rm	Consider the functional equation
	\begin{eqnarray}\label{Ex1}
	\frac{4}{5}f(x_1, x_2)-\frac{1}{10}f^2(x_1,x_2)=F(x_1,x_2),
	\end{eqnarray}
	where $F:[0,1]^2 \to [0,1]^2$ is defined by
	\begin{eqnarray*}
	F(x_1, x_2)=\left(\frac{x_1^2}{2}, \frac{x_1+x_2}{3}\right),\quad \forall (x_1, x_2)\in [0,1]^2.
	\end{eqnarray*}
 Then $\lambda_1=0.8<1$, $\lambda_2=-0.1<0$ such that $\lambda=\lambda_1+\lambda_2=0.7>0$ and $\frac{1}{\lambda} \sup F=\frac{1}{0.7}(\frac{1}{2}, \frac{2}{3})=(\frac{1}{1.4}, \frac{2}{2.1}) \in [0,1]^2$.
Also, $[0,1]^2$ is a convex complete  sublattice of $\mathbb{R}^2$ in the partial order $\preceq$ defined by $(x_1, x_2)\preceq (y_1, y_2)$ if
$x_k\le y_k$ for $k=1,2$. Further, it is easy to see that $F$ is order-preserving on $[0,1]^2$. Thus, all the
hypotheses of Theorem \ref{Thm1} are satisfied. Hence \eqref{Ex1} has a solution in $\mathcal{F}_{op}([0,1]^2)$.
}
\end{exmp}


\begin{exmp}
	{\rm	Consider the functional equation
		\begin{eqnarray}\label{Ex2}
		\frac{3}{4}f(x)-\frac{1}{5}f^2(x)=F(x),
		\end{eqnarray}
		where $F:[0,1] \to [0,1]$ is defined by $F(x)=\frac{x^3}{3}$.
	 Then $\lambda_1=0.75<1$, $\lambda_2=-0.2<0$ such that $\lambda=\lambda_1+\lambda_2=0.55>0$ and $\frac{1}{\lambda}\sup F=\frac{1}{\lambda} F(1)=\frac{1}{0.55}\cdot\frac{1}{3} =\frac{1}{1.65} \in [0,1]$.
		Also, $[0,1]$ is a convex complete sublattice of $\mathbb{R}$ in the usual order $\le$. Further,  $F$ is a continuous order-preserving  map on $[0,1]$. Thus, all the
		hypotheses of Theorem \ref{Thm2} are satisfied. Hence \eqref{Ex2} has a solution in $\mathcal{F}^{usc}_{op}([0,1])$.
	}
\end{exmp}

\begin{exmp}
	{\rm	Consider the functional equation
		\begin{eqnarray}\label{Ex4}
			\frac{7}{8}f(x)-\frac{1}{6}f^2(x)=F(x),
		\end{eqnarray}
		where $F:[0,1] \to [0,1]$ is defined by $F(x)=\frac{x^2}{4}$.
		Then $\lambda_1<1$, $\lambda_2<0$ such that $\lambda=\lambda_1+\lambda_2=\frac{17}{24}>0$ and $\frac{1}{\lambda}\sup F=\frac{1}{\lambda} F(1)=\frac{24}{17}\cdot\frac{1}{4} =\frac{6}{17} \in [0,1]$.
		Also,  $F$ is a order-preserving  map on $[0,1]$. Further, $F$ is  measurable, being a continuous map on $[0,1]$. Moreover, $\int_{0}^{1} |f|^3d\mu=\int_{0}^{1}\frac{x^6}{64} dx= \frac{1}{448}$, and therefore $F$ is $3$-integrable on $[0,1]$. Thus, all the
		hypotheses of Theorem \ref{Thm6} are satisfied. Hence \eqref{Ex4} has a solution in $\mathcal{F}^{3}_{op}([0,1])$.
	}
\end{exmp}

Remark that, since it is assumed that $0< \lambda_1\le 1$ in Theorems \ref{Thm1}, \ref{Thm2} and \ref{Thm6}
we cannot use these theorems to solve the iterative root problem $f^m=F$.
Besides, as remarked in the end of section \ref{sec22},
our current approach is not applicable to order-reversing cases. We leave these problems open for further investigation.



In addition to $\mathbb{R}^n$ and $\mathbb{R}$ considered in sections 4 and 5,
we can find more examples (\cite{Luxemburg1971}) of Riesz spaces across different branches of mathematics,
to which our results can be applied.

\begin{exmp}
{\rm
Let $\mathcal{G}(Y,\mathbb{R})$ be the real vector space of all real valued functions on
an arbitrary non-empty set $Y$ with the pointwise addition and the real scalar product.
Then it is a Riesz space in the order $\trianglelefteq$  defined by the inequality 
\begin{eqnarray}\label{pointwise order}
f\trianglelefteq g\quad \text{if}\quad f(x)\le g(x),~~ \text{for each}~x.
\end{eqnarray}
The subspace $\mathcal{G}_b(Y, \mathbb{R})$ of all real bounded functions on $Y$
is also a Riesz space in the same order.
In particular,
in the case that
$Y$ consists of $n$ points, where $n\in \mathbb{N}$,
the Riesz space
$\mathcal{G}(Y,\mathbb{R})$
 indeed
is the space $\mathbb{R}^n$ with coordinate-wise ordering.
In the case that $Y$ consists of a countably infinite number of points,
$\mathcal{G}(Y,\mathbb{R})$
is the sequence space $\mathbb{R}^\omega$ of all real sequences.
$\mathcal{G}_b(Y, \mathbb{R})$ is the subspace $l_\infty$ of all bounded real sequences.
Additionally,
the subspace ${\bf c}$ of all convergent sequences in $l_\infty$ and
the subspace ${\bf c}_0$ of all sequences in ${\bf c}$ whose limit is zero
are both Riesz spaces.
The spaces $l_\infty$, ${\bf c}$ and ${\bf c}_0$  are all Banach spaces
equipped with the norm $\|(x_k)\|_\infty=\sup\{|x_k|: k\in \mathbb{N}\}$.
}
\end{exmp}

Let $X$ be the Riesz space  $\mathcal{G}([0,1],\mathbb{R})$  and
		 $K=\mathcal{G}([0,1],[0,1])$.
		 Then it is easy to check that $K$ is a convex complete sublattice of $X$.
		Consider the functional equation
		\begin{eqnarray}\label{Ex3}
		\frac{3}{4}f(\phi)-\frac{1}{4}f^2(\phi)=F(\phi),
		\end{eqnarray}
	on $K$,	where $F:K \to K$ is defined by
		\begin{eqnarray*}
			F(\phi)=\frac{1}{5}({\rm id}+\phi),\quad \forall \phi \in K.
		\end{eqnarray*}
	 Then $\lambda_1=0.75<1$, $\lambda_2=-0.25<0$ such that $\lambda=\lambda_1+\lambda_2=0.5>0$ and $\frac{1}{\lambda} \sup F=\frac{1}{0.5}\cdot\frac{{\rm id}+\phi_b}{5} =\frac{2}{5}({\rm id}+\phi_b) \in K$, where $\phi_b$ is the constant map on $[0,1]$ defined by $\phi_b(x)=b$. Further, $F$ is  order-preserving
	 map on $K$.
	 Thus, all the
	hypotheses of Theorem \ref{Thm1} are satisfied. Hence \eqref{Ex3} has a solution  in $\mathcal{F}_{op}(K)$.


Similarly to the above,
we can consider the following.

\begin{exmp}
{\rm
 Let $l_p$ ($1\le p<\infty$) be the real vector space of all real sequences $(x_k)$ such that $\sum_{k=1}^{\infty}|x_k|^p<\infty$, with the coordinate-wise addition and the real scalar product. Then it is a Riesz space in the order $\preceq$
 defined by the inequality $(x_k)\preceq(y_k)$ if $x_k\le y_k$ for every $k\in \mathbb{N}$.
 Remark that the space $l_p$ ($1\le p<\infty$) is in fact a Banach space equipped with the norm
$\|(x_k)\|_p=\left(\sum_{k=1}^{\infty}|x_k|^p\right)^{\frac{1}{p}}$.
}
\end{exmp}

\begin{exmp}
{\rm
Let $\mathcal{C}(Y, \mathbb{R})$ be the real vector space of all real continuous functions on a nonempty topological space $Y$
with the pointwise addition and the real scalar product.
Then it is a Riesz space in the order $\trianglelefteq$ defined as in \eqref{pointwise order}.
The subspace $\mathcal{C}_b(Y, \mathbb{R})$ of all real bounded continuous functions on $Y$ is also a Riesz space
in the same order.
If $Y$ is a locally-compact space, then
the subspace $\mathcal{C}_c(Y,\mathbb{R})$ of all real continuous functions on $Y$ with compact support is a Riesz space.
}
\end{exmp}

\begin{exmp}
{\rm
Let $(Y, \Lambda,\mu)$ be a measure space, i.e.,
$\mu$ is a countably additive non-negative measure on the $\sigma$-field $\Lambda$ of subsets of the non-empty set $Y$.
Neglecting the measure zero sets,
we identify $\Lambda$ with the Boolean
measure algebra $\Lambda/\Lambda_0$, where $\Lambda_0$ is the ideal of measure zero sets.
Let ${\cal M}:={\cal M}(Y,\mu)$ denote the real vector space of all real $\mu$-almost everywhere finite valued $\mu$-measurable functions on $Y$,
with identification of $\mu$-almost equal functions, in the pointwise addition and the real scalar product.
Then ${\cal M}$ is a Riesz space in the order $\trianglelefteq$ defined by the inequality
$f\trianglelefteq g$ if $f(x)\le g(x)$ for $\mu$-almost every $x\in Y$.
More concretely,	
let $p$ be a real number such that $0<p<\infty$ and $L_p=L_p(Y,\mu)$ consist of all $f\in {\cal M}$ satisfying
$\int_{Y}^{}|f|^pd\mu<\infty.$
Then $L_p$ is a Riesz space in the order inherited from ${\cal M}$.
Further, the space $L_\infty$ consisting of all essentially bounded $f\in {\cal M}$
is also a Riesz space.
Remark that the spaces $L_p$, $1\le p\le \infty$, are all normed linear spaces equipped with
the norm
	\begin{eqnarray*}
		\|f\|_p=\left\{\begin{array}{ll}
		\left(\displaystyle\int_{Y}^{}|f|^pd\mu\right)^{\frac{1}{p}}	& \text{for}~1\le p<\infty,\\
		{\rm ess~sup}~|f(x)|& \text{for}~ p=\infty.
		\end{array}
		\right.
\end{eqnarray*}
}
\end{exmp}

\begin{exmp}
{\rm
Let $\mathcal{M}(Y)$ be the set of all finitely additive signed measures $\mu$ on the algebra $\mathcal{A}$
of subsets of a nonempty set $Y$ such that $\|\mu\|:=\sup\{|\mu(A)|:A\in \mathcal{A}\}$ is finite.
Then $\mathcal{M}(Y)$ is a Riesz space under the natural definitions of addition and
scalar product
in the order $\preceq$ defined by the inequality $\mu_1\preceq \mu_2$ if $\mu_1(A)\le \mu_2(A)$ for every $A\in \mathcal{A}$.
}
\end{exmp}

\begin{exmp}
{\rm
Let $H$ be a complex Hilbert space  with inner product $(\cdot, \cdot)$ and $\mathcal{H}$ denote the real vector space of
all bounded Hermitian operators on $H$. Then $\mathcal{H}$ is an ordered vector space
in the order $\preceq$
defined by the inequality $A\preceq B$ if $(Ax,x)\le (Bx, x)$ for each $x\in H$.
Although
$\mathcal{H}$ is not a Riesz space unless in the trivial case that $H$ is one-dimensional,
 many subspaces of $\mathcal{H}$ are Riesz spaces in the order inherited from $\mathcal{H}$.
Concretely,
let $\mathcal{D}$ be a nonempty subset of $\mathcal{H}$ such that
elements of $\mathcal{D}$ are mutually commuting (i.e., $AB=BA$ for all $A, B \in \mathcal{D}$),
$\mathcal{H}'$ consist of all elements of $\mathcal{H}$ which commute with $\mathcal{D}$
(i.e., each elements of $\mathcal{H}'$ commutes with each element of $\mathcal{D}$),
and $\mathcal{H}''$ consist of all elements of $\mathcal{H}$ which commute with $\mathcal{H}'$.
Then $\mathcal{H}''$ is a Riesz space such that $\mathcal{D}\subseteq \mathcal{H}'' \subseteq \mathcal{H}'$.
}
\end{exmp}

\begin{exmp}
{\rm
Let $\Omega$ be a region in the plane and $\Lambda$ be the ordered vector space of all functions $f(x,y)$ on $\Omega$
such that $f=u_1-u_2$ with
$u_1, u_2$
non-negative
and
harmonic
(i.e., $\Delta u_i=0, i=1,2$)
in $\Omega$,
where the order is the partial order $\trianglelefteq$ defined as in \eqref{pointwise order}.
Then $(\Lambda, \trianglelefteq)$ is a Riesz space.
}
\end{exmp}

Similarly to upper-semi-continuity considered in Corollary \ref{Thm5}, we can also consider another property $P$,
which can be, for instance, continuity, differentiability, integrability, measurability, or
the harmonicity in the sense of example 10
on a Riesz space $X$ in the above list of special cases.
Note that
Theorem \ref{Thm3} is true if $K$ is any sublattice of $X$
because we do not really use the convexity and completeness assumptions on $K$ while proving it.
Thus,
let $K$ be any sublattice of $X$ and
$$
\mathcal{F}^{P}_{op}(K):=\{f\in \mathcal{F}_{op}(K): f~\text{satisfies property P on}~K\}.
$$
Then Corollary \ref{Thm5} is indeed true whenever $\mathbb{R}^n$ is replaced by $X$ and $\mathcal{F}^{usc}_{op}(K)$  by $\mathcal{F}^{P}_{op}(K)$
since $\mathcal{F}^{P}_{op}(K) \subseteq \mathcal{F}_{op}(K)$.

\noindent {\bf Acknowledgment:}
The author Chaitanya Gopalakrishna is supported by Indian Statistical Institute
Bangalore in the form of a Visiting Scientist position through the J. C. Bose Fellowship
of Prof. B. V. Rajarama Bhat. The author is very grateful to Prof. S. Parameshwara Bhatta for his useful discussions.
The author Weinian Zhang would like to give his thanks to NNSFC for grants \# 11831012, \# 11771307 and \# 11821001.



%
%
%
%
%
%
%


\end{document}